\documentclass[journal]{IEEEtran}

\usepackage{graphics} 
\usepackage{wrapfig} 
\usepackage{subfigure} 
\usepackage{picinpar} 
\usepackage{rotating} 
\usepackage{color}
\usepackage[final,notref]{showkeys} 
\usepackage{epsfig}
\usepackage{amsmath}
\usepackage{txfonts} 
\usepackage{mathdots}
\usepackage{latexsym}
\usepackage{amssymb}
\usepackage{amscd}
\usepackage{trsym}
\usepackage{fancyhdr}
\usepackage{cite}
\graphicspath{{../texfigures}}

\newcommand{\beq}{\begin{equation}}
\newcommand{\eeq}{ \end{equation} }
\newcommand{\bea}{\begin{eqnarray}}
\newcommand{\eea}{\end{eqnarray}}
\newcommand{\ba}{\begin{array}}
\newcommand{\ea}{\end{array}}
\newcommand{\no}{\nonumber}

\newcommand{\halmos}{\hskip\textwidth minus\textwidth \EndProof}

\newcommand{\dbyd}[2]{\ensuremath{\tfrac{\mathrm{d} {#1}}{\mathrm{d}%
{#2}}}}

\newcommand{\real}[1]{\text{Re}\left\{#1\right\}}

\newcommand{\float}[3]{\ensuremath{{#1},\kern-0.12em{#2}\cdot 10^{#3}}}
\newcommand{\double}[2]{\ensuremath{{#1},\kern-0.12em{#2}}}

\newcommand{\abs}[1]{\ensuremath{\left\lvert{#1} \right\rvert}}
\newcommand{\R}{\ensuremath{\mathbb{R}}}

\newcommand{\TT}{\ensuremath{\mathbb{T}}}
\newcommand{\drm}{\ensuremath{\mathrm{d}}}

\newcommand{\htt}{\ensuremath{\hat{t}}}

\def\EndProof{\hfill \quad $\Box$  }
\newenvironment{proof}{\noindent{\bf Proof}: }{\halmos}

\newtheorem{theorem}{Theorem}[section]

\newtheorem{remark}{Remark}[section]

\begin{document}

\title{A Multirate Transient Simulation Technique for Circuits described by Integro-Differential Equations}

\author{Hans Georg Brachtendorf and	Kai Bittner
	\thanks{H.~G.~Brachtendorf
		is with the University
		of Applied Sciences of Upper Austria, Hagenberg, e-mail: brachtd@fh-hagenberg.at, kai.bittner@zeiss.com.}}

\maketitle

\begin{abstract}
The multirate partial differential equation (PDE) technique is considered
as one of the most powerful methods for the simulation of radio frequency
(RF) circuits. The method reformulates the underlying ordinary differential
algebraic equation (DAE) by a PDE. The DAE solution is obtained along a 
characteristic curve of the PDE. Numerical solution techniques comprise 
(envelope modulated) Harmonic Balance (HB), multistep finite difference (FD)
and spline/wavelet techniques. All techniques have disadvantages when mixed
lumped/distributed devices occur and the waveforms exhibit pulse shaped behavior. 
Furthermore, communication systems are modeled and simulated by the Equivalent 
Complex Baseband (ECB) approach. Signals and impulse responses are described
 by 
their baseband counterparts, employing the Hilbert transformation. A coupled 
system/circuit simulation requires therefore novel numerical techniques for 
the solution of integro-differential equations.
\end{abstract}


\IEEEpeerreviewmaketitle

\section{Introduction}

The multirate partial differential equation (PDE) technique \cite{BWL+96,Roy01,Roy97,NL96,PG02} is 
a powerful tool for simulating radio frequency (RF) circuits. The well-known 
multirate Harmonic Balance (HB) method is a special numerical solution technique
 when all signal waveforms
are in steady state and the signals are expanded in a trigonometric or Fourier basis with two or more
incommensurable fundamental frequencies. The advantage of HB lies in the fact, that distributed devices
are characterized by their impulse response or vice versa transfer characteristic
$h(t) \leftrightarrow H(\omega)$, whereas lumped elements are characterized in the time domain.
Moreover, when the signals behave quasi-sinusoidal, the trigonometric basis leads to small systems of
equations. 

However, in modern RF circuits the latter assumption is not valid anymore, since the signals
are often pulse modulated with short slew rate. Therefore the convergence of a Fourier
series is rather poor and one observes even worse the Gibbs phenomenon.
Therefore pure time domain techniques with
variable step sizes or spline/wavelet methods \cite{BiBra14a} lead to a more compact representation of the
waveforms, which require however a time domain representation of the devices. 

The multirate PDE technique can handle 
not only steady state problems, but initial value problems too.
Though envelope modulated HB techniques have been proposed, which always
presuppose a slowly varying envelope, which is a priori not guaranteed, pure time domain or spline/wavelet
techniques lead to better approximations of the waveforms \cite{BiBra14a}. However,  existing time domain 
and spline/wavelet methods
lack from the fact that they are limited to circuits with lumped devices, i.e.\ a time domain
representation of their characteristics.
The latter limitation can be overcome in part by model order reduction (MOR) techniques, which lead
however to a large increase of unknowns, since the transfer characteristic must be
approximated sufficiently accurate within a predefined frequency range. Another drawback appears in the
coupled simulation of circuits and communication systems. Communication systems are modeled in the
equivalent complex baseband  (ECB) \cite{proakis2001digital}, i.e., all signals and impulse responses which
exhibit bandpass (modulated) behavior are represented by their lowpass or baseband counterparts.
Indeed, the ECB technique can be considered as a multirate simulation approach in its own right,
which however presupposes linear time-invariant
systems, characterized by their impulse responses. To this end,
the multirate PDE technique must be generalized for integro-differential equations.

In this letter we develop a general framework for coupling the ECB method with the multirate PDE technique by
preserving the advantages of both techniques, i.e., the time domain or spline/wavelet representation
of pulse shaped signals, accurate representation of nonlinear devices including industry standards
for semiconductors, modeling of distributed devices and/or systems by their impulse responses or
transfer functions. Furthermore, the technique shall not only be applicable to steady state
but also to initial value problems for keeping the full capabilities of the PDE technique.

In Section~2 the multirate PDE technique is revisited, for more details we refer
the reader to \cite{BWL+96,BiBra14a,Roy01,NL96,PG02}, and further references within these papers. 
The 
necessary generalization to integro-differential equations is considered in Sec.~3. The Sec.~4
deals with non-autonomous circuits with a priori known and fixed center frequency and Sec.~5
with implementation aspects. Finally Sec.~6 deals with the relation of the novel
method to the classical ECB technique.

\section{The multirate PDE method}

We consider the differential algebraic equation (DAE) of dimension $N$
resulting from, e.g., the Modified Nodal Analysis (MNA) and the lumped device
models
\beq
\tfrac{d}{dt}q\big(x(t)\big)
        + i\big(x(t)\big) = s(t),\, x(0) = x_0 \label{eq.1}
\eeq
with a solution in the radio frequency range $x_{RF}$ exhibiting a 
high center or carrier
frequency and a slowly varying envelope, the baseband signal.
The multirate PDE formulation \cite{BWL+96,Roy01,NL96,PG02} 
essentially splits the envelope from the
carrier waveform\footnote{For ease of presentation only the bi-variate 
case is considered here.}
\begin{align}
&\tfrac{\partial}{\partial \tau} q\big(\hat{x}(\tau,\htt)\big)
+ \dbyd{(\omega(\tau)\,\tau)}{\tau}\,
  q\big(\hat{x}(\tau,\htt)\big)
+i\big(\hat{x}(\tau,\htt)\big)=\hat{s}\big(\tau,\htt\big), \no\\
&\quad \hat{x}(0,\Theta) = x_0(\Theta) \label{eq.2}
\end{align}
with instantaneous angular frequency $\omega(\tau)$, $\hat{x}: \R \times \mathbb{T}
\rightarrow \R^N$, where $\mathbb{T}$ represents the unit circle which means that the
solution is $2\,\pi$ periodic in the variable $\htt$. Moreover, we assume
 initial conditions $x_0(\Theta), \Theta \, \in \mathbb{T}$, $i,\,q: \R^N
\rightarrow \R^N$ and a unique
 solution $\hat{x}_{RF}$. 
 For driven circuits the frequency is known a priori, for the autonomous
 case the instantaneous frequency must be estimated \cite{BBL12a}.
 The characteristic curves, parametrized by $\Theta$, of the PDE are given by
 $(\tau,\,\htt) = (t,\omega(t)\,t + \Theta)$ with phase function
 \beq
  \varphi(t) := \omega(t)\,t \label{eq.2a}
 \eeq 
The solutions of the
 underlying ordinary DAE along the characteristic curves are denoted by 
 $x_\Theta(t) = \hat{x}_{RF}(t,\omega(t)\,t + \Theta)$
 for the initial conditions $x_0(\Theta)$.
  It is known that the PDE and DAE solutions coincide
along the specific characteristic curve with $\Theta=0$
iff $\hat{x}(0,0) = x_0$ and $\hat{s}\big(t,\omega(t)\,t\big) = s(t)$,
i.e., along the curve $(\tau,\,\htt) = (t,\omega(t)\,t)$.

\section{Integro-differential algebraic equations}

In RF circuit design, the lumped model assumption in \eqref{eq.1} above is
often not valid anymore.
Instead, devices such as transmission lines, filters, baluns etc.\ are characterized and
measured in the frequency domain. The lumped DAE formulation must therefore be
extended by a convolution integral, i.e.,
\beq
\tfrac{d}{dt}q\big(x(t)\big)
        + i\big(x(t)\big) + \int_{-\infty}^t h(t-t')\,x(t')\,\drm t' = s(t),
        \, x_\Theta(0) = x_0(\Theta) \label{eq.3}
\eeq
with impulse response $h$, in short $h * x$, where $h(t)\equiv 0 \, \forall t<0$.
Moreover, the convolution integral is commutative, i.e.,
\[
 (Y\,x)(t) := \int_{-\infty}^t h(t-t')\,x(t')\,dt' = \int_0^{\infty} h(t')\,x(t-t')\,\drm t'
\]
where $Y$ represents a linear operator $Y: \,(L^1(\R))^N \rightarrow (L^1(\R))^N$ and
$L^1$ is the space of absolutely integrable functions.
For the generalization of the bi-variate PDE formulation \eqref{eq.2}
 corresponding to \eqref{eq.3} we pursue the following ansatz
\begin{align}
	&\tfrac{\partial}{\partial \tau} q\big(\hat{x}(\tau,\htt)\big)
	+\dbyd{(\omega(\tau)\,\tau)}{\tau}\,
	 q\big(\hat{x}(\tau,\htt)\big)
	+i\big(\hat{x}(\tau,\htt)\big) \no\\
	& \quad +   (\hat{Y}\,\hat{x})(\tau,\htt)
	    = \hat{s}\big(\tau,\htt\big), 
	\quad \hat{x}(0,\Theta) = x_0(\Theta) \label{eq.4}
\end{align}
with a linear operator $\hat{Y}: \,(L^1(\R \times \TT))^N \rightarrow (L^1(\R \times \TT))^N$.

From the uni-variate convolution in \eqref{eq.3} one obtains along a characteristic curve
\begin{align*}
    y_\Theta(t) &:= \int_{-\infty}^t h(t-\tau')\,x_\Theta(\tau')\,\drm \tau'
    = \int_0^\infty h(\tau')\,x_\Theta(t-\tau')\,\drm \tau'
\end{align*}
and since the DAE and PDE solutions coincide along these curves
$
 x_\Theta(t) = \hat{x}(t,\,\omega(t)\,t+\Theta)
$,
it follows with \eqref{eq.2a}  that
\begin{align*}
    y_\Theta(t) &= \int_{-\infty}^t h(t-\tau')\,
       \hat{x}(\tau',\,\omega(\tau')\,\tau'+\Theta)\,\drm \tau'\\
       &= \int_0^\infty h(\tau')\,\hat{x}(t-\tau',\,\varphi(t-\tau')+\Theta)\,\drm \tau'
\end{align*}
Along the characteristic curve
\[
 y_\Theta(t) = \hat{y}(t,\,\omega(t)\,t+\Theta)
\]
must hold for the output signal too, since the DAE and PDE solutions must coincide

\begin{theorem}
For unmodulated waveforms $\omega(t)=\omega_0=\text{const.}$ we obtain
\begin{align}
\hat{y}(\tau,\htt) &=  \int_{-\infty}^\tau 
 {h}(\tau-\tau')\,\hat{x}(\tau',\htt-\omega_0\,(\tau-\tau'))\,\drm\tau' \no\\
&= \int_0^{\infty} 
  {h}(\tau')\,\hat{x}(\tau-\tau',\htt-\omega_0\,\tau')\,\drm\tau'
 \label{eq.5b}
\end{align}
\end{theorem}
\begin{proof}
    The theorem follows immediately from inserting the characteristic
    curve $(\tau,\,\htt) = (t,\,\omega_0\,t+\Theta)$ in \eqref{eq.5b}.
\end{proof}

\begin{remark}
One can see from (\ref{eq.5b})
that the convolution is performed along the characteristic curve.
\end{remark}

In what follows we investigate sufficient conditions for a bi-variate convolution
integral, pursuing the ansatz
\begin{align}
\hat{y}(\tau,\htt) &= \hat{h}(\tau,\htt)\; \circledast \;
\hat{x}(\tau,\htt) \no\\
 &:=  \int_{-\infty}^\tau  \int_\TT
\hat{h}(\tau-\tau',\htt-t')\,\hat{x}(\tau',t')\,\drm t'\,\drm\tau' \no\\
&= \int_0^{\infty} \int_\TT
\hat{h}(\tau',t')\,\hat{x}(\tau-\tau',\htt-t')\,\drm t'\,\drm\tau'
 \label{eq.5}
\end{align}
The inner integral is defined as cyclic convolution on the torus $\TT$, i.e., an integration on the unit
circle modulo $2\,\pi$, and
an a-cyclic convolution in $\tau$ as outer integral.
A necessary condition on $\hat{h}$ is derived next.

\begin{theorem}
 \label{theo.1}
 The DAE \eqref{eq.3} and PDE solutions \eqref{eq.4} coincide along the
 characteristic curves if the equalities
	\begin{align}
	 & h(t-\tau')\, \hat{x}(\tau',\,\omega(\tau')\,\tau'+\Theta) \no\\
	 &\qquad = \int_\TT
	      \hat{h}(t-\tau',\,\omega(t)\,t+\Theta-t')\,\hat{x}(\tau',\,t')\,\drm t'
	      \label{eq.theo.1}
	\end{align}
	likewise
	\begin{align}
	 & h(\tau')\,\hat{x}(t-\tau',\,\varphi(t-\tau')+\Theta) = \no\\
	 &\qquad \int_\TT\,
	      \hat{h}(\tau',\,t')\,\hat{x}(t-\tau',\,\omega(t)\,t+\Theta-t')\,\drm t'
	      \label{eq.theo.2}
	\end{align}
 hold.
\end{theorem}

\begin{proof}
From the uni-variate convolution one obtains
\begin{align*}
    y_\Theta(t) &= \int_{-\infty}^t h(t-\tau')\,x_\Theta(\tau')\,\drm \tau'
    = \int_0^\infty h(\tau')\,x_\Theta(t-\tau')\,\drm \tau'
\end{align*}
and since the DAE and PDE solutions coincide along the characteristic curves
$
 x_\Theta(t) = \hat{x}(t,\,\omega(t)\,t+\Theta)
$,
it follows that
\begin{align*}
    y_\Theta(t) &= \int_{-\infty}^t h(t-\tau')\,
       \hat{x}(\tau',\,\omega(\tau')\,\tau'+\Theta)\,\drm \tau'\\
       &= \int_0^\infty h(\tau')\,\hat{x}(t-\tau',\,\varphi(t-\tau')+\Theta)\,\drm \tau'
\end{align*}
Since
\[
 y_\Theta(t) = \hat{y}(t,\,\omega(t)\,t+\Theta)
\]
must hold along the characteristic curve for the output signal too, 
we obtain from the bi-variate convolution
\begin{align*}
    y_\Theta(t) &= \int_{-\infty}^t \int_\TT
      \hat{h}(t-\tau',\,\omega(t)\,t+
      \Theta-t')\,\hat{x}(\tau',\,t')\,\drm t'\,\drm \tau'\\
      &= \int_0^\infty \int_\TT\,
      \hat{h}(\tau',\,t')\,\hat{x}(t-\tau',\,\omega(t)\,t+\Theta-t')\,\drm t'\,\drm \tau'
\end{align*}
From comparison one concludes that
\begin{subequations}\label{eq.4a}
\begin{align}
 &  h(t-\tau')\, \hat{x}(\tau',\omega(\tau')\,\tau'+\Theta) = \no \\
 &\qquad  \int_\TT
      \hat{h}(t-\tau',\omega(t)\,t+\Theta-t')\,\hat{x}(\tau',t')\,\drm t' \\
 & \text{or equivalently} \no \\
 & h(\tau')\,\hat{x}(t-\tau',\,\varphi(t-\tau')+\Theta) = \no \\
 &\qquad  \int_\TT\,
      \hat{h}(\tau',\,t')\,\hat{x}(t-\tau',\,\omega(t)\,t+\Theta-t')\,\drm t'
\end{align}
\end{subequations}
holds.
\end{proof}


\section{The case $\omega(\tau) = \omega_0 = \text{const.}$}

In what follows the special case $\omega(\tau) = \omega_0 = \text{const.}$ is considered,
i.e.\ $\varphi(t) = \omega_0\,t$, which is valid for non-autonomous systems where the center
frequency is known a priori.  We restrict here the signals
to continuous functions, i.e., $Y: \,(C(\R))^N \rightarrow (C(\R))^N$.

\begin{theorem}
    \label{theo.2}
    The equalities \eqref{eq.4a} of Theorem~\ref{theo.1} are fulfilled if
	\beq
	 \hat{h}(\tau,\,\hat{t}) = h(\tau)\,\delta(\hat{t} - \omega_0\,\tau) \label{70}
	\eeq
	where $\delta$ is the periodic Dirac distribution.
\end{theorem}

\begin{proof}
    From the sift property of Dirac's $\delta$ 
    distribution\footnote{Sift property: $\delta_T * x = x_T$.} it follows immediately
\begin{align*}
     &\int_\TT\,
      \hat{h}(\tau',\,t')\,\hat{x}(t-\tau',\,\omega_0\,t+\Theta-t')\,\drm t'\\
	 & \qquad = \int_\TT 
	 \delta(t'-\omega_0\,\tau')\,h(\tau')\,
	 \hat{x}(t-\tau',\,\omega_0\,t+\Theta-t')\,\drm t'\\
	 &\qquad =  h(\tau')\,\hat{x}(t-\tau',\,\omega_0\,(t-\tau')+\Theta)
\end{align*}
Likewise
\begin{align*}
    &\int_\TT
      \hat{h}(t-\tau',\omega_0\,t+\Theta-t')\,\hat{x}(\tau',t')\,\drm t'  \\
    &\qquad  = \int_\TT \delta(\omega_0\,\tau'+\Theta-t')\, h(t-\tau')\,
    \hat{x}(\tau',t')\,\drm t'  \\
    &\qquad = h(t-\tau')\, \hat{x}(\tau',\omega_0\,\tau'+\Theta)
\end{align*}
\end{proof}

Hence \eqref{eq.5} reads
\begin{align}
\hat{y}(\tau,\htt) &= \int_{-\infty}^\tau 
 {h}(\tau-\tau')\,\hat{x}(\tau',\htt - \omega_0\,(\tau-\tau'))\,\drm\tau' \no\\
&= \int_0^{\infty} 
  {h}(\tau')\,\hat{x}(\tau-\tau',\htt-\omega_0\,\tau')\,\drm\tau'
 \label{eq.5a}
\end{align}

Along the characteristic curve $(t,\,\omega_0\,t + \Theta)$ one gets 

\begin{align}
 y_\Theta(t) &= \hat{y}(t,\,\omega_0\,t+\Theta)
 =  \int_{-\infty}^\tau 
 {h}(t-\tau')\,\hat{x}(\tau',\omega_0\,\tau'  + \Theta)\,\drm\tau' \no\\
&= \int_0^{\infty} 
  {h}(\tau')\,\hat{x}(t-\tau',\omega_0\,(t-\tau') + \Theta)\,\drm\tau' 
  = h(t) * x_\Theta(t) \label{60}
\end{align}

\section{Implementation}

For the case of unmodulated waveforms or signals with compact spectra,
centered around a carrier frequency $f_0=\frac{\omega_0}{2\,\pi}=\text{const.}$,  
a numerically efficient method is developed next.
   
Let $\hat{x}(\tau,\htt)$ be 
expanded in a Fourier series in the variable $\htt$, i.e.\
\beq
 \label{eq.6}
 \hat{x}(\tau,\htt) = 
  \sum_{k} \hat{x}_k(\tau)\,\mathrm{e}^{{}\jmath \, k\, \htt} 
\eeq
The DAE solutions are $x_\Theta(t) = \hat{x}(t,\omega_0\,t + \Theta)$.
Moreover, $\hat{x}_{-k}(\tau) = \hat{x}^*_k(\tau)$ holds, where the asterisk
represents the complex conjugate.
The so called envelopes $\hat{x}_k(\tau)$ exhibit in practical multirate
simulations lowpass (band limited) behavior when the initial conditions
$x_0(\Theta)$ of \eqref{eq.4} are chosen properly, however
band limitation is not explicitly required here.
From Theorem~\ref{theo.2} and (\ref{eq.5b},\ref{eq.6}) one gets
\begin{align}
    \hat{y}(\tau,\htt) &= \sum_k \int_0^\infty h(\tau')\,
     \hat{x}_k(\tau-\tau')\, 
     \mathrm{e}^{\jmath \,k\, (\htt - \omega_0 \,\tau')}\,  \drm \tau' \no\\
     &= \sum_k \mathrm{e}^{{}\jmath \,k\, \htt} 
     \int_0^\infty h(\tau')\, \mathrm{e}^{-\jmath \, k\, \omega_0 \,\tau'}\,
     \hat{x}_k(\tau-\tau')\,\drm \tau' \label{8a}
\end{align}
The Fourier transformation of the convolution integral reads
\beq
\int_0^\infty h(\tau')\, \mathrm{e}^{-\jmath \, k\, \omega_0 \,\tau'}\,
     \hat{x}_k(\tau-\tau')\,\drm \tau' \leftrightarrow 
     H(\omega+k\, \omega_0) \, \hat{X}_k(\omega)
     \label{eq.8}
\eeq
where $h \leftrightarrow H,\, \hat{x}_k \leftrightarrow \hat{X}_k$
are the Fourier transform pairs.

\begin{theorem}
    \label{theo.3}
    Let $\omega(\tau)=\omega_0=\text{const.}$
    The equalities \eqref{eq.theo.1} and \eqref{eq.theo.2}
    are fulfilled if
	\beq
	 \hat{h}(\tau,\,\hat{t}) = h(\tau)\,\delta(\hat{t} - \omega_0\,\tau) =
	 \sum_k \frac{1}{2 \pi} h(\tau)\, \mathrm{e}^{-\jmath \, k\, \omega_0 \,\tau}
	 \,\mathrm{e}^{\jmath \,k\, \htt} \label{80}
	\eeq
\end{theorem}

\begin{proof}
 From the expansion \eqref{eq.6} and the orthogonality of the exponential function
 $<\mathrm{e}^{\jmath \,k\, t},\,\mathrm{e}^{\jmath \,l\, t}> \, = 2\,\pi\,\delta_{kl}$,
 where $< \cdot, \cdot >$ inner product and $\delta_{kl}$ is the Kronecker $\delta$ symbol,
 one obtains from \eqref{eq.5} immediately the integral \eqref{8a}.
\end{proof}

\begin{remark}
    In RF circuit design the transfer function often shows bandpass behavior,
    i.e.\ the spectrum is concentrated around $k=\pm 1$ and with \eqref{8a}
    \[
        \hat{y}(\tau,\htt) \approx \int_0^\infty h(\tau')\, 
        \mathrm{e}^{-\jmath \,( \omega_0 \,\tau' - \htt)}\,\hat{x}_1(\tau-\tau') +
     \mathrm{e}^{\jmath \,( \omega_0 \,\tau' - \htt)}\,\hat{x}_{-1}(\tau-\tau')\,\drm \tau'
    \]
    Since the modulated impulse response 
    $h(t)\,\mathrm{e}^{\mp\jmath \, \omega_0 \,t}
    \leftrightarrow H(\omega\pm \omega_0)$ exhibits lowpass, i.e.\ smooth, behavior, 
    the convolution integral \eqref{eq.8} can be calculated efficiently by a coarse
    discretization by, e.g., multistep integration methods.
    The equation \eqref{eq.8} shows similarity to the method of the equivalent 
    baseband  \cite{proakis2001digital} which is used in communication engineering
    for treating bandpass modulated signals efficiently, as considered in the next
    Section.
    However, the calculation of the analytical signal via the Hilbert transformation
    is not required.
\end{remark}

\begin{theorem}
   \label{theo.4.2}
   Let $\omega(\tau)=\omega_0=\text{const.}$  The equality
 	\beq
	 \hat{h}(\tau,\,\hat{t}) = 
	 \sum_k \frac{1}{2 \pi} h(\tau)\, \mathrm{e}^{-\jmath \,  k\, \omega_0 \,\tau}
	 \,\mathrm{e}^{\jmath \,k\, \htt} = h(\tau)\, \delta(\omega_0\,\tau - \htt) \no
	\eeq
holds.
\end{theorem}

\begin{proof}
 Since
 \[
   \delta(t) = \frac{1}{T}\sum_k \mathrm{e}^{-\jmath \, 2\pi \, k\, t/T}
 \]
 holds on $\TT$, one gets from \eqref{80} the identity
 \begin{align}
 \hat{h}(\tau,\,\hat{t}) &= 
	 \frac{1}{2 \pi} h(\tau)\sum_k  \mathrm{e}^{-\jmath \, k \,\omega_0 \, \tau}
	 \,\mathrm{e}^{\jmath \,k\, \htt} = \frac{1}{\omega_0} h(\tau)\,
	 \delta\left(\tau - \frac{\htt}{\omega_0} \right)	 \no\\
	 &= h(\tau)\sum_k \delta(\omega_0\,\tau - \htt)	 \no
 \end{align}
 where the latter equality results from the identity $\delta(a\,t - t_0) = \frac{1}{\abs{a}} 
 \delta\left(t - \frac{t_0}{a}\right)$.
\end{proof}


\section{Relation to the ECB method}
The ECB method is briefly revisited. For more details we refer to \cite{proakis2001digital}.
We consider (complex) baseband signals or envelopes $x_{BB}(t)$ and Radio Frequency
(RF) or bandpass signals $x_{RF}(t)$ and their relation.

\begin{remark}
    The ECB method as derived here ensures that the envelope and the RF signals
    have equal power. The scaling differs in literature.
\end{remark}

\subsection{Transformation of signals}

\subsubsection{Calculation of the baseband from the RF signal}
1.\ step: Hilbert transformation and the analytical signal: defining
\[
H_{\mathrm{RF}}^{\pm}(\omega) = \left(1 {\pm} \mathrm{sign}\left(\omega\right)\right)
\leftrightarrow h_{\mathrm{RF}}^{\pm}(t) = \delta(t) \pm \frac{\jmath}{\pi t} 
\]
where $\jmath = \sqrt{-1}$, $\mathrm{sign}(\cdot)$ the signum function,
one obtains the analytical signal by
\begin{align*}
  \no X_{\mathrm{RF}}^{\pm}\left(\omega\right)
  &=  H_{\mathrm{RF}}^{\pm}(\omega) \cdot X_{\mathrm{RF}}(\omega) = 
  \left(1 \pm \mathrm{sign}\left(\omega\right)\right) \cdot X_{\mathrm{RF}}\left(\omega\right)
     \leftrightarrow\\
  x_{\mathrm{RF}}^\pm(t) &=  x_{\mathrm{RF}}\left(t\right) \pm \jmath \, 
  x_{\mathrm{RF}}\left(t\right) * \frac{1}{\pi t}
    =:  x_{\mathrm{RF}}\left(t\right) \pm \jmath \, \mathcal{H}\left\{x_{\mathrm{RF}}\left(t\right)\right\} 
\end{align*}
where $\mathcal{H}\left\{x_{\mathrm{RF}}\left(t\right)\right\}$ is referred to as 
the Hilbert transformation.

2. step: calculation of the complex baseband or lowpass signal $x_{BB}(t)$ by modulation/frequency shift
of the analytical signal

\begin{align*}
  \no X_{BB}\left(\omega\right) &= \frac{1}{\sqrt{2}} \, X_{\mathrm{RF}}^{+}\left(\omega + \omega_0\right)\\
  & = \frac{1}{\sqrt{2}} \, \left(1 + \mathrm{sign}\left(\omega + \omega_0\right)\right) \cdot X_{\mathrm{RF}}\left(\omega + \omega_0\right) \leftrightarrow\\
  \no x_{BB}\left(t\right) &= \frac{1}{\sqrt{2}} \, x_{\mathrm{RF}}^{+}\left(t\right) 
  \cdot \mathrm{e}^{-\jmath \, \omega_0 \, t} \\
  & = \frac{1}{\sqrt{2}} \, \left(x_{\mathrm{RF}}\left(t\right) + \jmath \,
   \mathcal{H}\left\{x_{\mathrm{RF}}\left(t\right)\right\}\right) 
   \cdot \mathrm{e}^{-\jmath \, \omega_0 \, t}
\end{align*}

with inphase and quadrature phase components
\[
 X_{BB}(f) \leftrightarrow x_{BB}(t) = x_I(t) + \jmath\, x_Q(t)
\]

The reverse operation leads to the RF signal by
$x_{\mathrm{RF}}\left(t\right) = 
  \sqrt{2} \, \real{x_{BB}\left(t\right) \cdot \mathrm{e}^{\jmath \, \omega_0 \, t}}
$

\subsection{Transformation of LTI systems}

\begin{remark}
    The transformation of LTI systems is equivalent to signals besides
    a multiplicative constant which ensures that the output envelope and
    RF signals have equal powers.
\end{remark}

From $y_{\mathrm{RF}}\left(t\right)= x_{\mathrm{RF}}\left(t\right) * h_{\mathrm{RF}}\left(t\right)$
an ECB impulse response $h_{BB}(t)$ is required, where 
$y_{BB}\left(t\right) = x_{BB}\left(t\right)  * h_{BB}\left(t\right)$ is the ECB signal of $y_{\mathrm{RF}}$.
 The transformation of the impulse response
 $h_{\mathrm{RF}}(t)$ into the baseband reads
\begin{align*}
  \no H_{BB}\left(\omega\right) &= \frac{1}{2} \, 
  \left(1 + \mathrm{sign}\left(\omega + \omega_0\right)\right) 
  \cdot H_{\mathrm{RF}}\left(\omega + \omega_0\right) \leftrightarrow\\
  \no h_{BB}\left(t\right) &= \frac{1}{2} \, \left(h_{\mathrm{RF}}\left(t\right) + \jmath \, \mathcal{H}\left\{h_{\mathrm{RF}}\left(t\right)\right\}\right) \cdot
   \mathrm{e}^{-\jmath \, \omega_0 \, t}
\end{align*}

\begin{theorem}
    The ECB signal $y_{BB}$ of $y_{RF} = h_{RF} * x_{RF}$ is given by
    $y_{BB} = h_{BB} * x_{BB}$, where $x_{BB}$ and $h_{BB}$ are the equivalent baseband signals
    of $x_{RF}$ and $h_{RF}$ given above.
\end{theorem}

\begin{proof}
 The proof can be found in \cite{proakis2001digital}.
\end{proof}

From \eqref{eq.8} and $k=1$ one can consider $\hat{X}_1(\omega) := X_{BB}(\omega)$ as the envelope or equivalent
baseband waveform and $H(\omega + \omega_0) := H_{BB}(\omega)$
 is the equivalent baseband transfer function
of the integro-differential multirate PDE. The method is therefore compatible with the ECB technique
and enables the coupled simulation of circuits and communication systems.

\subsection*{Conclusion}

The multirate PDE method has been generalized to integro-differential algebraic equations. 
Such
an integral term is required when the RF circuit contains, e.g.,
linear subcircuits or distributed devices described by their impulse responses
or vice versa in the frequency domain by their transfer functions.
The multirate solutions must coincide with the ordinary integro-DAE 
solutions
along the characteristic curves of the multirate integro-PDE. 
An ansatz via a multirate convolution integral or vice versa
a multivariate impulse resonse is considered.
First, by Theorem~\ref{theo.1}, a necessary condition is derived which must be fulfilled
by the multivariate impulse response. Second, a sufficient condition \eqref{70}
for the case of a fixed center frequency, which is the typical case in RF design, is deduced.
Third,
an alternative condition \eqref{80} is derived, which allows a run-time
efficient implementation and is similar to
the method of the Equivalent Complex Baseband method
in communication engineering, but does
not require a cumbersome Hilbert transformation. Moreover, in Theorem~\ref{theo.4.2},
it is shown that both formulations of the multivariate impulse responses are equivalent.
The technique enables time domain and spline/wavelet methods, 
Fourier or harmonic balance methods, coupled system and 
circuit/device simulation,
both in time and frequency domain.

\enlargethispage{\baselineskip}

\section{Acknowledgment}
In memory of Thomas Brazil, University of Dublin,
who posed this problem.

\bibliographystyle{plain}
\bibliography{./literature}  

\end{document}